\documentclass{amsart}
\usepackage{amsfonts,amscd,amsthm,amsgen,amsmath,amssymb}
\usepackage[all]{xy}
\usepackage{epsfig,color}
\usepackage[vcentermath]{youngtab}
\newtheorem{theorem}{Theorem}[section]
\newtheorem{lemma}[theorem]{Lemma}
\newtheorem{corollary}[theorem]{Corollary}

\theoremstyle{definition}
\newtheorem{definition}[theorem]{Definition}

\theoremstyle{remark}
\newtheorem{remark}[theorem]{Remark}

\numberwithin{equation}{section}

\begin{document}

\title{A note on the Nielsen realization problem for K3 surfaces}
\author{David Baraglia}
\address{School of Mathematical Sciences, The University of Adelaide, Adelaide SA 5005, Australia}
\email{david.baraglia@adelaide.edu.au}

\author{Hokuto Konno}
\address{RIKEN iTHEMS, Wako, Saitama 351-0198, Japan}
\email{hokuto.konno@riken.jp}

\begin{abstract}
We will show the following three theorems on the diffeomorphism and homeomorphism groups of a $K3$ surface.
The first theorem is that the natural map $\pi_{0}(Diff(K3)) \to Aut(H^{2}(K3;\mathbb{Z}))$ has a section over its image.
The second is that, there exists a subgroup $G$ of $\pi_{0}(Diff(K3))$ of order two over which there is no splitting of the map $Diff(K3) \to \pi_{0}(Diff(K3))$, but there is a splitting of $Homeo(K3) \to \pi_{0}(Homeo(K3))$ over the image of $G$ in $\pi_{0}(Homeo(K3))$, which is non-trivial.
The third is that the map $\pi_{1}(Diff(K3)) \to \pi_{1}(Homeo(K3))$ is not surjective.
Our proof of these results is based on Seiberg-Witten theory and the global Torelli theorem for $K3$ surfaces.
\end{abstract}


\date{\today}



\maketitle


\section{Introduction}

In this paper, we shall show several theorems on the diffeomorphism and homeomorphism groups of a $K3$ surface combining results obtained from Seiberg-Witten theory and from the global Torelli theorem for $K3$ surfaces.

We denote by $X$ the underlying smooth $4$-manifold of a $K3$ surface, $Diff(X)$ the group of diffeomorphisms with the $\mathcal{C}^\infty$-topology and $Mod(X) = \pi_0(Diff(X))$ the mapping class group. Let $L$ denote the lattice $H^2(X ; \mathbb{Z})$ with its natural intersection pairing and $Aut(L)$ the group of automorphisms of $L$. Let $\Gamma \subset Aut(L)$ denote the image of $Mod(X)$ in $Aut(L)$. By a result of Kreck \cite{kre}, $\Gamma$ is the group of pseudo-isotopy classes of diffeomorphisms of $X$. It is known that $\Gamma$ is the index two subgroup of automorphisms of $L$ which preserve orientation on $H^+(X)$ \cite{mat, don}. From the definition of $\Gamma$ we have a surjective map
\[
p : Mod(X) \to \Gamma.
\]
Let $Homeo(X)$ be the group of homeomorphisms of $X$ with the $\mathcal{C}^0$-topology. By the work of Freedman and Quinn \cite{fre, qui}, it is known that the natural map $\pi_0(Homeo(X)) \to Aut(L)$ is an isomorphism. The groups that we have just introduced are related to one another by a commutative diagram:
\begin{equation*}\xymatrix{
Diff(X) \ar[d]^-i \ar[r] & Mod(X) \ar[r]^-p \ar[d] & \Gamma \ar[dl] \\
Homeo(X) \ar[r] & Aut(L) & 
}
\end{equation*}

In this note we prove the following theorems concerning these groups, related to the Nielsen realization problem for $K3$ surfaces.

\begin{theorem}\label{thm:section}
There exists a splitting $s : \Gamma \to Mod(X)$ of $p$.
\end{theorem}

Theorem~\ref{thm:section} is shown using the global Torelli theorem.
Our usage of the global Torelli theorem is basically due to work of Giansiracusa~\cite{gia} and of Giansiracusa, Kupers, and Tshishiku~\cite{gkt}.

Combing Theorem~\ref{thm:section} with the adjunction inequality, which is an input from Seiberg-Witten theory, we shall give the negative answer to the $4$-dimensional Nielsen realization problem.
The original Nielsen realization problem asks whether every finite subgroup of the mapping class group of an oriented closed surface can be realized as a subgroup of the diffeomorphism group.
Kerckhoff~\cite{ker} solved this original problem in the affirmative.
The following Theorem~\ref{thm:nielsen} tells that the analogous statement in dimension $4$ does not hold in general:

\begin{theorem}\label{thm:nielsen}
There is a subgroup of $Mod(X)$ of order $2$ which does not lift to a subgroup of order $2$ in $Diff(X)$. However, the image of this subgroup in $Aut(L)$, which is non-trivial, lifts to a subgroup of order $2$ in $Homeo(X)$.
\end{theorem}

This gives an example where the smooth Nielsen realization problem for $K3$ can not be solved, but the corresponding continuous Nielsen realization problem can. A completely different example where the smooth Nielsen-type realization problem for $K3$ can not be solved was recently constructed by Giansiracusa, Kupers and Tshishiku \cite{gkt}. In their example, non-realizability is demonstrated using non-vanishing of certain generalized Miller-Morita-Mumford classes. These are rational cohomology classes on $BDiff(X)$ and it follows that they can not be used to detect failure to lift a {\em finite} subgroup of $Mod(X)$ to $Diff(X)$. Thus Theorem \ref{thm:nielsen} does not follow from the constructions of \cite{gkt}.

The last result in this paper is a comparison between $Diff(X)$ and $Homeo(X)$.
As preceding results, Donaldson~\cite{don} showed that the map $i_{\ast} : \pi_{0}(Diff(X)) \to \pi_{0}(Homeo(X))$ induced from the inclusion $i : Diff(X) \to Homeo(X)$ is not surjective.
In \cite{bako}, the authors proved that at least one of the following two statements is true: $\pi_{0}(Diff(X)) \to \pi_{0}(Homeo(X))$ is not injective, or $\pi_{1}(Diff(X)) \to \pi_{1}(Homeo(X))$ is not surjective.
Ruberman showed in ~\cite{rub1} that $\pi_{0}(Diff(M)) \to \pi_{0}(Homeo(M))$ is not injective for some $4$-manifolds $M$, and this was generalized to other $4$-manifolds in \cite{bako1} by the authors.
However, these results can not be applied to $M=X$, a $K3$ surface.
All of these results are based on gauge theory.
In particular, the authors' result in \cite{bako} is a consequence of Seiberg-Witten theory for families.
In this paper, combining such a gauge-theoretic result in \cite{bako} with an input from the global Torelli theorem, we show:

\begin{theorem}\label{thm:nonsmooth}
The induced map
\[
i_* : \pi_1(Diff(X)) \to \pi_1(Homeo(X))
\]
is not surjective.
\end{theorem}

\begin{corollary}
The group $\pi_1(Homeo(X))$ is non-trivial.
\end{corollary}

\noindent{\bf Acknowledgments}.
D. Baraglia was financially supported by the Australian Research Council Discovery Project DP170101054.
H. Konno was supported by JSPS Grant-in-Aid for Scientific Research on Innovative Areas ``Discrete Geometric Analysis for Materials Design" (Grant Number 17H06461).

\section{Einstein metrics on $K3$}

To prove Theorem \ref{thm:section} we need to recall some facts concerning the Torelli theorem for $K3$ surfaces \cite{bura, loo}. Let $I$ be a complex structure on $X$ with trivial canonical bundle so that $(X,I)$ is a complex $K3$ surface. Then $H^{2,0}(X,I)$ is a $1$-dimensional subspace of $H^2(X ; \mathbb{C})$. Suppose that $z$ spans $H^{2,0}(X,I)$. From Hodge theory it is known that $\langle z , z \rangle = 0$ and $\langle z , \overline{z} \rangle > 0$. Writing $z = x + iy$, where $x,y \in H^2(X ; \mathbb{R})$ one has $\langle x , x \rangle = \langle y , y \rangle > 0$ and $\langle x , y \rangle = 0$. Let $P_{(X,I)} = \mathbb{R}x + \mathbb{R}y$ be the span of $x$ and $y$. Then $P_{(X,I)}$ is a positive definite, oriented $2$-plane in $H^2(X ; \mathbb{R})$. Define the root system of $(X,I)$ to be
\[
\Delta_{(X,I)} = \{ \delta \in H^2(X ; \mathbb{Z}) \; | \; \delta \perp P_{(X,I)}, \; \; \langle \delta , \delta \rangle = -2 \}.
\]
It is known that for each $\delta \in \Delta_{(X,I)}$, either $\delta$ or $-\delta$ is represented by an effective divisor. Any class $\kappa \in H^2(X ; \mathbb{R})$ of a K\"ahler form for $(X,I)$ is orthogonal to $P_{(X,I)}$ and has positive inner product with the class of an effective divisor. Hence $\langle \kappa , \delta \rangle \neq 0$ for all $\delta \in \Delta_{(X,I)}$. Therefore $\kappa$ is in a connected component of
\[
\mathcal{K}_{(X,I)} = \{ u \in H^2(X ; \mathbb{R}) \; | \; u \perp P_{(X,I)}, \; \; \langle u , u \rangle > 0, \; \; \langle u , \delta \rangle \neq 0 \text{ for all } \delta \in \Delta_{(X,I)} \}.
\]
The set of K\"ahler classes for $(X,I)$ is convex, hence connected and therefore lies in a distinguished connected component of $\mathcal{K}_{(X,I)}$. We call this component the {\em K\"ahler chamber} of $(X,I)$. The theorem of Burns-Rapoport \cite{bura} states that if $X,X'$ are two $K3$ surfaces and $\phi : H^2(X ; \mathbb{Z}) \to H^2(X' ; \mathbb{Z})$ is an isometry sending $H^{2,0}(X)$ to $H^{2,0}(X')$ and sending the K\"ahler chamber of $X$ to the K\"ahler chamber of $X'$, then there is a unique isomorphism $f : X \to X'$ inducing $\phi$.

Now let $g$ be an Einstein metric on $X$. It is known that any such metric is in fact hyperk\"ahler \cite[Chapter 12. K]{bes}. Let $I,J,K$ be a hyperk\"ahler triple of complex structures for $g$ and $\omega_I, \omega_J, \omega_K$ the corresponding K\"ahler forms. Then $\{ \omega_I , \omega_J , \omega_K \}$ defines an oriented basis for $H^+_g(X)$. Since the hyperk\"ahler triple $(I,J,K)$ is determined by $g$ up to an $SO(3)$ transformation, it follows that the orientation induced on $H^+_g(X)$ depends only on the metric $g$.

\begin{lemma}\label{lem:orient}
Let $g,g'$ be two Einstein metrics on $X$ such that $H^+_g(X) = H^+_{g'}(X)$ as unoriented subspaces of $H^2(X ; \mathbb{R})$. Then $g$ and $g'$ induce the same orientation on $H^+_g(X)$.
\end{lemma}
\begin{proof}
Without loss of generality assume $g,g'$ both have unit volume. Suppose that $H^+_g(X) = H^+_{g'}(X)$, but that $g,g'$ induce opposite orientations. We will derive a contradiction. Let $(I,J,K)$ be a hyperk\"ahler triple for $g$ with corresponding K\"ahler forms $\omega_I, \omega_J, \omega_K$. Consider the complex $K3$ surface given by $(X,I)$. Then $H^{2,0}(X,I)$ is spanned by $\omega_J + i\omega_K$ and the corresponding $2$-plane $P_{(X,I)}$ is spanned by $\omega_J, \omega_K$. By rotating $(I,J,K)$ if necessary, we can assume that each $\delta \in H^2(X ; \mathbb{Z}))$ with $\langle \delta , \delta \rangle = -2$ is not orthogonal to $P_{(X,I)}$. Hence $\Delta_{(X,I)}$ is empty. Then since $P_{(X,I)}^\perp$ has signature $(1,19)$ it follows that
\[
\mathcal{K}_{(X,I)} = \{ u \in P_{(X,I)}^\perp \; | \; \langle u , u \rangle > 0 \}
\]
has exactly two connected components, which we denote as $\mathcal{K}_{(X,I)}^+ , \mathcal{K}_{(X,I)}^-$. Let us assume that $\mathcal{K}^+_{(X,I)}$ is the K\"ahler chamber for $(X,I)$.

Now let $(I',J',K')$ denote a hyperk\"ahler triple for $g'$. Since $H^+_g(X) = H^+_{g'}(X)$ but with the opposite orientation, it follows that $(I',J',K')$ can be chosen so that
\[
[ \omega_I] = -[\omega_{I'}], \quad [ \omega_J] = [\omega_{J'}], \quad [ \omega_K] = [\omega_{K'}],
\]
where $[ \; . \; ]$ denotes the underlying cohomology class. It follows that $H^{2,0}(X,I) = H^{2,0}(X,I')$, $\mathcal{K}_{(X,I)} = \mathcal{K}_{(X,I')}$ and that the K\"ahler component of $(X,I')$ is $\mathcal{K}_{(X,I)}^-$.

Let $\phi : H^2(X ; \mathbb{Z}) \to H^2(X ; \mathbb{Z})$ be the isometry $\phi(x) = -x$. Then $\phi$ sends $H^{2,0}(X,I)$ to $H^{2,0}(X,I')$ and sends $\mathcal{K}_{(X,I)}^+$ to $\mathcal{K}_{(X,I)}^-$, so by the theorem of Burns-Rapoport, $\phi$ is induced by a diffeomorphism $f : X \to X$. But this would mean that $f$ is a diffeomorphism which reverses orientation on $H^+(X)$, which is known to be impossible. Hence $g,g'$ must induce the same orientation on $H^+_g(X)$.
\end{proof}

\begin{proof}[Proof of Theorem \ref{thm:section}]
The proof is an adaptation of an argument given in \cite{gkt} (see also \cite{gia}). Let $Ein$ denote the space of all Einstein metrics on $X$ with the $\mathcal{C}^\infty$ topology and with unit volume. We have that $Diff(X)$ acts on $Ein$ by pullback. Let $TDiff(X)$ denote the subgroup of $Diff(X)$ acting trivially on $H^2(X ; \mathbb{Z})$. So we have a short exact sequence
\[
1 \to TDiff(X) \to Diff(X) \to \Gamma \to 1.
\]
It is a consequence of the global Torelli theorem for $K3$ surfaces that $TDiff(X)$ acts freely (and properly) on $Ein$ (see \cite[Lemma 4.4, Lemma 4.5]{gia}). Let
\[
T_{Ein} = Ein/TDiff(X)
\]
be the quotient ($T_{Ein}$ is an analogue of Teichm\"uller space for $K3$ surfaces). Let $Gr_3( \mathbb{R}^{3,19})$ denote the Grassmannian of positive definite $3$-planes in $\mathbb{R}^{3,19}$. There is a period map
\[
P : T_{Ein} \to Gr_3( \mathbb{R}^{3,19}).
\]
Defined as follows. Fix an isometry $H^2(X ; \mathbb{R}) \cong \mathbb{R}^{3,19}$. Then $P$ sends an Einstein metric $g$ the $3$-plane $H^+_g(X)$. Let
\[
\Delta = \{ \delta \in H^2(X ; \mathbb{Z}) \; | \; \delta^2 = -2 \}
\]
and set
\[
W = \{ \tau \in Gr_3(\mathbb{R}^{3,19}) \; | \; \tau^\perp \cap \Delta = \emptyset \}.
\]
The Grassmannian $Gr_3(\mathbb{R}^{3,19})$ is a contractible manifold of dimension $3 \cdot 19 = 57$. For each $\delta \in \Delta$, the subset
\[
A_\delta = \{ \tau \in Gr_3(\mathbb{R}^{3,19}) \; | \; \delta \in \tau^\perp \}
\]
is a codimension $3$ embedded submanifold. Then $W = Gr_3(\mathbb{R}^{3,19}) \setminus \bigcup_{\delta \in \Delta} A_\delta$. A transversality argument implies that $W$ is connected and simply-connected. To be more precise, let $\gamma : S^1 \to W$ be a loop in $W$ based at some point $x_0 \in W$. Since $Gr_3(\mathbb{R}^{3,19})$ is a simply-connected smooth manifold, there exists a smooth homotopy $\gamma_t$ of loops based at $x_0$ from $\gamma_0 = \gamma$ to the constant loop $\gamma_1 = x_0$. By \cite[Theorem 2.5]{hir}, we can assume $\gamma_t$ is transverse to each of the countably many submanifolds $\{ A_\delta \}_{\delta \in \Delta}$. But this means the image of $\gamma_t$ is disjoint from the $A_\delta$. Hence $\gamma$ is contractible as a based loop in $W$.

It can be shown that the period map $P$ takes values in $W$ and moreover the global Torelli theorem implies that the period map $P : T_{Ein} \to W$ is a homeomorphism. More precisely, $P$ is a local homeomorphism by the local Torelli theorem, surjectivity of $P$ follows from \cite{loo} and injectivity from the discussion given in \cite[Chapter 12. K]{bes}. Note that in \cite[Chapter 12. K]{bes}, the period map is taken with values in $Gr_3^+(\mathbb{R}^{3,19})$, the Grassmannian of positive oriented $3$-planes in $\mathbb{R}^{3,19}$. However, because of Lemma \ref{lem:orient}, this distinction is irrelevant. Let 

\begin{equation}\label{equ:mein}
\mathcal{M}_{Ein} = Ein \times_{Diff(X)} EDiff(X)
\end{equation}
be the homotopy quotient (this is a slight variant of the moduli space $\mathcal{M}_{Ein}$ defined in \cite{gia}). Since $TDiff(X)$ acts freely and properly on $Ein$, one finds that
\[
\mathcal{M}_{Ein} = T_{Ein} \times_{\Gamma} E\Gamma.
\]
We have seem that $T_{Ein}$ is homeomorphic to $W$. Hence $T_{Ein}$ is connected and simply-connected. By the long exact sequence of homotopy groups associated to the fibration $T_{Ein} \to \mathcal{M}_{Ein} \to B\Gamma$, we get
\[
1 = \pi_1(T_{Ein}) \to \pi_1(\mathcal{M}_{Ein}) \to \pi_1(B\Gamma) \to \pi_0(T_{Ein}) = 1.
\]
Hence the natural projection map $\mathcal{M}_{Ein} \to B\Gamma$ induces an isomorphism $\pi_1(\mathcal{M}_{Ein}) \to \pi_1(B\Gamma) = \Gamma$. But from Equation (\ref{equ:mein}), it follows that the map $\mathcal{M}_{Ein} \to B\Gamma$ factors as
\[
\mathcal{M}_{Ein} \to BDiff(X) \to B\Gamma.
\]
It follows that the induced map $s : \pi_1(\mathcal{M}_{Ein}) \to \pi_1(BDiff(X)) = \pi_0(Diff(X)) = Mod(X)$ is a splitting of $p: Mod(X) \to \Gamma$.
\end{proof}

\section{Proof of Theorem \ref{thm:nielsen}}

Recall that $X = K3$ is homeomorphic to $3(S^2 \times S^2) \# 2(-E_8)$, where $-E_8$ denotes the compact, simply-connected topological $4$-manifold with intersection form minus the $E_8$-lattice. We construct a continuous involution on $X$ as follows. Let $f_0 : S^2 \times S^2 \to S^2 \times S^2$ be given by
\[
f(x,y) = (y , x).
\]
Note that $f$ has fixed points. Thus we can form the equivariant connected sum $3(S^2 \times S^2)$, summing together three copies of $(S^2 \times S^2 , f_0)$. Now by attaching two copies of $-E_8$ we obtain a continuous involution $f : X \to X$. Let $\phi \in Aut(L)$ denote the isometry induced by $f$. One easily checks that $\phi \in \Gamma$. Set $\tilde{\phi} = s(\phi) \in Mod(X)$, where $s : \Gamma \to Mod(X)$ is the section of Theorem \ref{thm:section}. Then $\tilde{\phi}$ generates a subgroup of $Mod(X)$ of order $2$.

Theorem \ref{thm:nielsen} is an immediate consequence of the following:
\begin{theorem}
Let $g \in Diff(X)$ be a lift of $\tilde{\phi}$ to a diffeomorphism. Then $g$ is not an involution.
\end{theorem}
\begin{proof}
We assume $g$ is an involution and derive a contradiction. We will make use of results and terminology of \cite{ed}. The action of $g$ on $H^2(X ; \mathbb{Z})$ can be decomposed into three types: trivial, cyclotomic and regular. Let $(t,c,r)$ denote the number of each such summand. Hence
\[
t + c + 2r = b_2(X) = 22.
\]
The action of $g$ on $H^2(X ; \mathbb{Z})$ is given by $\phi$. A simple calculation shows that
\[
(t,c,r) = (0,0,11).
\]
It follows that $g$ does not act freely on $X$ (a free involution would have type $(0,2,r)$).

Recall that an involution is called {\em even} if it lifts to an involution on the spin bundle and called {\em odd} otherwise. We claim that $g$ is odd. In fact, from the definition of $\phi$, one easily computes that
\[
b_+^{\mathbb{Z}_2}(X) = 3, \quad b_-^{\mathbb{Z}_2}(X) = 8, \quad \sigma(X)^{\mathbb{Z}_2} = -5.
\]
If $g$ were even, the $G$-signature theorem would give $\sigma(X)^{\mathbb{Z}_2} = -8$, a contradiction. So $g$ is odd.

Since $g$ is odd, the fixed point set consists of embedded surfaces $\Sigma_1, \dots , \Sigma_k$, $k>0$, where each $\Sigma_i$ is oriented (since $X$ is spin). Let $\Sigma_i$ have genus $g_i$. Then from \cite{ed},
\[
t = 2k-2, \quad c = 2(g_1+ \cdots + g_k).
\]
But $t=0$ and $c=0$, so $k = 1$ and $g_1 = 0$. The $G$-signature theorem implies
\[
\sigma(X)^{\mathbb{Z}_2} = \frac{1}{2} \sigma(X) + \frac{1}{2}( [\Sigma_1]^2 )
\]
hence
\[
-5 = -8 + \frac{1}{2}[\Sigma_1]^2
\]
and so
\[
[\Sigma_1]^2 = 6.
\]
But this contradicts the adjunction inequality (see \cite[Theorem 11]{law}), since $\Sigma_1$ has genus $0$. So such a $g$ can not exist.
\end{proof}

\begin{remark}
It is interesting to note that the only place in which we used that $g$ was smooth, and not just locally linear, is in the adjunction inequality.
\end{remark}

\section{Constructing families over $T^2$}\label{sec:constructing}

In this section we describe how families (continuous or smooth) of $K3$ surfaces can be constructed over the $2$-torus $B = T^2$.
\begin{definition}
By a {\em continuous family of $K3$ surfaces over $B$}, we mean a topological fibre bundle $\pi : E \to B$ with fibres homeomorphic to $K3$. Thus $E$ is the associated fibre bundle of a principal $Homeo(X)$-bundle. We say that a continuous family $E \to B$ is {\em smoothable} with fibres diffeomorphic to $X$, if the underlying principal $Homeo(X)$-bundle can be reduced to a principal $Diff(X)$-bundle.
\end{definition}

As explained in \cite[\textsection 4.2]{bako}, it follows from a result of M\"uller-Wockel \cite{muwo} that $E$ is smoothable with fibres diffeomorphic to $X$ if and only if $E$ admits the structure of a smooth manifold such that $\pi : E \to B$ is a submersion and the fibres of $E$ with their induced smooth structure are diffeomorphic to $X$.

We are interested in studying principal $G$-bundles on $T^2$, where $G = Homeo(X)$ or $G = Diff(X)$. Regard $T^2$ as a CW complex with two $1$-cells and a single $2$-cell. The $1$-skeleton is a wedge of two circles. A principal $G$-bundle $P \to B$ over $B$ can be constructed in two steps:
\begin{itemize}
\item{Construct $P$ over the $1$-skeleton. For each loop, we need an element of $\pi_0(G)$, which describes how to identify the fibres of $P$ over the endpoints of the $1$-cell being attached. Let $x_1,x_2 \in \pi_0(G)$ denote these elements.}
\item{Extend $P$ over the $2$-cell. For this we need $P$ to be trivial over the boundary of the $2$-cell. The attaching map of the $2$-cell is the commutator map. In other words, $P$ can be extended if and only if $x_1,x_2$ commute as elements of $\pi_0(G)$.}
\end{itemize}

From these remarks and obstruction theory, we conclude the following:
\begin{itemize}
\item{Let $P$ be a principal $G$-bundle over $T^2$. The restriction of $P$ to the $1$-skeleton of $T^2$ defines elements $x_1,x_2 \in \pi_0(G)$. If $P'$ is a second principal $G$-bundle over $T^2$ with corresponding elements $x'_1,x'_2 \in \pi_0(G)$, then a necessary condition for $P,P'$ to be isomorphic is that $x_i = x'_i$ for $i=1,2$.}
\item{Suppose $P,P'$ are two principal $G$-bundles over $T^2$ and $x_i = x'_i$ for $i=1,2$. Then there is a difference obstruction in $H^2( T^2 ; \pi_1(G))$ which is the obstruction to extending an isomorphism $P \to P'$ over the $2$-cell.}
\end{itemize}
Note that the group $H^2( T^2 ; \pi_1(G))$ is not necessarily isomorphic to $\pi_1(G)$, because we have to consider $\pi_1(G)$ as a local system on $T^2$. However $H^2( T^2 ; \pi_1(G))$ is easily seen to be isomorphic to a quotient of $\pi_1(G)$ (use the cellular model for cohomology).

\section{Proof of Theorem \ref{thm:nonsmooth}}

In \cite[\textsection 4.2]{bako} we construct commuting homeomorphisms $f_1, f_2 : X \to X$. Let $E \to T^2$ be the mapping torus. This defines a continuous $K3$ family over $T^2$ or equivalently a principal $Homeo(X)$-bundle over $T^2$. Moreover it is shown that this family is not smoothable \cite[Theorem 4.24]{bako}.

For $i=1,2$, let $\rho_i = (f_i)_* \in Aut(L)$ denote the induced automorphisms of $M$. Then it is easily verified that $\rho_1,\rho_2 \in \Gamma$. For $i=1,2$, let $g_i = s(\rho_i) \in Mod(X)$. Choose actual diffeomorphisms $\tilde{g}_1, \tilde{g}_2$ representing $g_1,g_2$. Then $\tilde{g}_1, \tilde{g}_2$ commute up to a smooth isotopy. Choose one such smooth isotopy. From this data we can construct a smoothable family over $T^2$ and a corresponding principal $Diff(X)$ bundle in exactly the manner described in Section \ref{sec:constructing}. Let $E' \to T^2$ denote the underlying continuous family. Then:
\begin{itemize}
\item{$E$ and $E'$ are isomorphic over the $1$-skeleton of $T^2$. This is because $g_1,g_2$ and $f_1,f_2$ induce the same elements of $Aut(L) = \pi_0( Homeo(X))$.}
\item{$E$ and $E'$ are not isomorphic over $T^2$ because $E'$ is smoothable but $E$ is non-smoothable.}
\end{itemize}
Thus $E$ and $E'$ differ by a non-trivial element $\mathcal{O} \in H^2( T^2 ; \pi_1(Homeo(X)))$. Let
\[
Q = Homeo(X) \times_{Diff(X)} EDiff(X)
\]
be the homotopy quotient. Theorem \ref{thm:nonsmooth} will follow if we can show that image of $\mathcal{O}$ under the natural map
\begin{equation}\label{equ:natmap}
H^2( T^2 ; \pi_1(Homeo(X)) ) \to H^2( T^2 ; \pi_1(Q))
\end{equation}
is non-zero. But note that $Q$ can be identified with the homotopy fibre of $BDiff(X) \to BHomeo(X)$:
\[
Q \to BDiff(X) \to BHomeo(X).
\]
By obstruction theory, there are a sequence of obstructions to smoothing $E$ which take values in $H^j( T^2 ; \pi_{j-1}(Q) )$. Since $E$ is smoothable on the $1$-skeleton of $B$, the non-smoothability of $E$ means that the obstruction in $H^2( T^2 ; \pi_1(Q))$ is non-trivial. The obstruction class is easily seen to be the image of $\mathcal{O}$ under the natural map (\ref{equ:natmap}). Therefore, the image of $\mathcal{O}$ under this map is non-zero.


\bibliographystyle{amsplain}

\end{document}